\documentclass[12pt]{amsart}
\usepackage{amsmath,amssymb,amsbsy,amsfonts,latexsym,amsopn,amstext,cite,
                                               amsxtra,euscript,amscd,bm,mathabx,mathrsfs, abraces}
\usepackage{url}
\usepackage[colorlinks,linkcolor=blue,anchorcolor=blue,citecolor=blue,backref=page]{hyperref}
\usepackage{color}
\usepackage{graphics,epsfig}
\usepackage{graphicx}
\usepackage{float} 
\usepackage[english]{babel}
\usepackage{mathtools}
\usepackage{todonotes}
\usepackage{url}
\usepackage[colorlinks,linkcolor=blue,anchorcolor=blue,citecolor=blue,backref=page]{hyperref}
\usepackage{soul}

\usepackage[norefs,nocites]{refcheck}

\hypersetup{breaklinks=true}

\usepackage[norefs,nocites]{refcheck}
\usepackage[english]{babel}
\begin{document}

\newtheorem{thm}{Theorem}
\newtheorem{lem}[thm]{Lemma}
\newtheorem{claim}[thm]{Claim}
\newtheorem{cor}[thm]{Corollary}
\newtheorem{prop}[thm]{Proposition} 
\newtheorem{definition}[thm]{Definition}
\newtheorem{question}[thm]{Open Question}
\newtheorem{conj}[thm]{Conjecture}
\newtheorem{rem}[thm]{Remark}
\newtheorem{prob}{Problem}

\def\ccr#1{{#1}}
\def\cco#1{{#1}}
\def\ccg#1{{#1}}

\newtheorem{ass}[thm]{Assumption}

\newtheorem{lemma}[thm]{Lemma}

\newcommand{\GL}{\operatorname{GL}}
\newcommand{\SL}{\operatorname{SL}}
\newcommand{\lcm}{\operatorname{lcm}}
\newcommand{\ord}{\operatorname{ord}}
\newcommand{\Tr}{\operatorname{Tr}}
\newcommand{\Span}{\operatorname{Span}}
\newcommand{\Spec}{\operatorname{Spec}}
\newcommand{\charpoly}{\operatorname{char}}

\numberwithin{equation}{section}
\numberwithin{thm}{section}
\numberwithin{table}{section}

\def\vol {{\mathrm{vol\,}}}
\def\squareforqed{\hbox{\rlap{$\sqcap$}$\sqcup$}}
\def\qed{\ifmmode\squareforqed\else{\unskip\nobreak\hfil
\penalty50\hskip1em\null\nobreak\hfil\squareforqed
\parfillskip=0pt\finalhyphendemerits=0\endgraf}\fi}

\def \balpha{\bm{\alpha}}
\def \bbeta{\bm{\beta}}
\def \bgamma{\bm{\gamma}}
\def \blambda{\bm{\lambda}}
\def \bchi{\bm{\chi}}
\def \bphi{\bm{\varphi}}
\def \bpsi{\bm{\psi}}
\def \bomega{\bm{\omega}}
\def \btheta{\bm{\vartheta}}
\def \bmu{\bm{\mu}}
\def \bnu{\bm{\nu}}

\newcommand{\bfxi}{{\boldsymbol{\xi}}}
\newcommand{\bfrho}{{\boldsymbol{\rho}}}

\def\cA{{\mathcal A}}
\def\cB{{\mathcal B}}
\def\cC{{\mathcal C}}
\def\cD{{\mathcal D}}
\def\cE{{\mathcal E}}
\def\cF{{\mathcal F}}
\def\cG{{\mathcal G}}
\def\cH{{\mathcal H}}
\def\cI{{\mathcal I}}
\def\cJ{{\mathcal J}}
\def\cK{{\mathcal K}}
\def\cL{{\mathcal L}}
\def\cM{{\mathcal M}}
\def\cN{{\mathcal N}}
\def\cO{{\mathcal O}}
\def\cP{{\mathcal P}}
\def\cQ{{\mathcal Q}}
\def\cR{{\mathcal R}}
\def\cS{{\mathcal S}}
\def\cT{{\mathcal T}}
\def\cU{{\mathcal U}}
\def\cV{{\mathcal V}}
\def\cW{{\mathcal W}}
\def\cX{{\mathcal X}}
\def\cY{{\mathcal Y}}
\def\cZ{{\mathcal Z}}
\def\Ker{{\mathrm{Ker}}}

\def\sA{{\mathscr A}}

\def\NmQR{N(m;Q,R)}
\def\VmQR{\cV(m;Q,R)}

\def\Xm{\cX_m}

\def \A {{\mathbb A}}
\def \B {{\mathbb A}}
\def \C {{\mathbb C}}
\def \F {{\mathbb F}}
\def \G {{\mathbb G}}
\def \L {{\mathbb L}}
\def \K {{\mathbb K}}
\def \Q {{\mathbb Q}}
\def \R {{\mathbb R}}
\def \Z {{\mathbb Z}}

\def \fA{\mathfrak A}
\def \fC{\mathfrak C}
\def \fL{\mathfrak L}
\def \fR{\mathfrak R}
\def \fS{\mathfrak S}

\def \fUg{{\mathfrak U}_{\mathrm{good}}}
\def \fUm{{\mathfrak U}_{\mathrm{med}}}
\def \fV{{\mathfrak V}}
\def \fG{\mathfrak G}
\def \f{\mathfrak G}

\def\e{{\mathbf{\,e}}}
\def\ep{{\mathbf{\,e}}_p}
\def\eq{{\mathbf{\,e}}_q}

 \def\\{\cr}
\def\({\left(}
\def\){\right)}
\def\fl#1{\left\lfloor#1\right\rfloor}
\def\rf#1{\left\lceil#1\right\rceil}

\def\Im{{\mathrm{Im}}}

\def \oF {\overline \F}

\newcommand{\pfrac}[2]{{\left(\frac{#1}{#2}\right)}}

\def \Prob{{\mathrm {}}}
\def\e{\mathbf{e}}
\def\ep{{\mathbf{\,e}}_p}
\def\epp{{\mathbf{\,e}}_{p^2}}
\def\em{{\mathbf{\,e}}_m}

\def\Res{\mathrm{Res}}
\def\Orb{\mathrm{Orb}}

\def\vec#1{\mathbf{#1}}
\def \va{\vec{a}}
\def \vb{\vec{b}}
\def \vc{\vec{c}}
\def \vs{\vec{s}}
\def \vu{\vec{u}}
\def \vv{\vec{v}}
\def \vw{\vec{w}}
\def\vlam{\vec{\lambda}}
\def\flp#1{{\left\langle#1\right\rangle}_p}

\def\mand{\qquad\mbox{and}\qquad}

\title[Number of characteristic polynomials]{Number of characteristic polynomials of matrices with bounded height}


  \author[L. M{\'e}rai]{L{\'a}szl{\'o} M{\'e}rai}
 \address{L.M.:  Department of Computer Algebra, ELTE
Eötvös Loránd University, Budapest,
Hungary}
 \email{merai@inf.elte.hu}


\author[I. E. Shparlinski] {Igor E. Shparlinski}
\address{School of Mathematics and Statistics, University of New South Wales, Sydney NSW 2052, Australia}
\email{igor.shparlinski@unsw.edu.au}

\begin{abstract}  We consider the set  $\cM_n\(\Z; H\)$ of $n\times n$-matrices with 
integer elements of size at most $H$ and obtain  upper and lower bounds on the number of
distinct irreducible characteristic polynomials which correspond to these matrices and thus
on the number of distinct eigenvalues of  these matrices. In particular, we improve
some results of  A.~Abrams, Z.~Landau, J.~Pommersheim and N.~Srivastava (2022).
\end{abstract}



\maketitle


\section{Introduction}

For a positive integer $n$  we use $\cM_n\(\Z\)$ to  denote the
set of all  $n\times n$ matrices with integer entries. 
Furthermore, for  an integer $H\ge 1$ we use $\cM_n\(\Z; H\)$  to denote the
set  of matrices
\[
A = (a_{ij})_{i,j=1}^n \in \cM_n\(\Z\)
\]  
with integer entries of size $|a_{ij}| \le H$.  In particular, $\cM_n\(\Z; H\)$ is 
of cardinality $\# \cM_n\(\Z; H\) = \(2H +1\)^{n^2}$. 

Motivated by recent work of  Abrams, Landau, Pommersheim and  Srivastava~\cite{ALPS}, 
we consider the following quantities 
 \begin{itemize}
\item $\cI_n(H)$, which is the set of all   irreducible characteristic polynomial corresponding 
to matrices $A \in \cM_n\(\Z;H\)$
;
\item $\Spec(n, H)$, which is the set of distinct eigenvalues of matrices $A \in \cM_n\(\Z;H\)$. 
\end{itemize}
We note that in turn the work~\cite{ALPS} has been motivated by applications to stability analysis of some numerical
linear algebra algorithms. However we believe that counting problems of this type with integral matrices 
are of independent interest, we refer to~\cite{HOS} for some recent results and further references. 

In particular, we recall that by~\cite[Theorem~1.3]{ALPS}, for any integer $k \ge 0$ and 
$n = 2^k+1$ we have 
\begin{equation}
\label{eq:ALPS Bound}
\# \Spec(n, H) \ge \frac{n}{5^{n}}  H^{(n-1)^2/4} . 
\end{equation}
Since it is easy to show that $\# \Spec(n, H)$ is a monotonic function of $n$, for a fixed $n$ of the form $n=2^k$ (which is the worst 
choice) this leads to a lower bound of order $H^{n^2/16}$. This result is based on showing that the set  $\cI_n(H)$ contains
sufficiently many polynomials. 

Here we supplement the argument of~\cite{ALPS} with some other ideas, which allows us to improve the bound~\eqref{eq:ALPS Bound}
and also extend it to arbitrary odd integer $n = 2k+1$.  

 In fact, as in~\cite{ALPS} our lower bounds apply to a smaller sets
$\cI_n^+(H)$ and $\Spec^+(n, H)$ 
which are defined as $\cI_n(H)$ and $\Spec(n, H)$ but with respect to the matrices 
$A \in \cM_n^+\(\Z;H\)$, where $\cM_n^+\(\Z;H\)\subseteq \cM_n\(\Z;H\)$ is the
set of matrices with integer entries $a_{ij} \in \{0, \ldots, H\}$.

We obtain lower bounds on the cardinality $\cI_n^+(H)$ which we then combine with the trivial inequality 
$$
\# \Spec(n, H) \ge \# \Spec^+(n, H) \ge n \# \cI_n^+(H).
$$
In particular, for any  fixed $n$,  our lower bounds give $(n-1)^2/4$ as the exponent of $H$ 
improving on $n^2/16$, which is implied by the argument of~\cite{ALPS} in the worst case. 
We also
significantly improve the coefficients depending on $n$.

We start with a result which applies when $H$ is much bigger that $n$. 
However, note that do not assume that  $n$ is fixed, and in the implied constant 
in the `$O$'-symbol is absolute. 

\begin{thm}\label{thm:InH-asym} For any odd $n\ge 11$ and $H\to \infty$, 
we have
$$
 \# \cI_n^+(H) \ge  \frac{1}{n}   H^{(n-1)^2/4} + O\( H^{(n-1)^2/4 - (n-9)/(n-5)} \).
 $$
\end{thm}

Theorem~\ref{thm:InH-asym} is nontrivial when $H \ge Cn$ for some absolute constant $C>0$. We also have 
the following fully explicit estimate, which applies to any $n\ge 5$ and $H \ge 3$.
It seems that no such estimate  has been available prior this work.

\begin{thm}\label{thm:InH-expl} 
 For  any  odd integers $n\ge 5$ and $H\ge 3$, 
we have
$$
 \# \cI_n^+(H) \ge    \frac{1}{4n} \( H-1\)^{(n-1)^2/4} . 
 $$
\end{thm}

We conclude with the observation that the question of estimating $ \# \cI_n^+(H) $ is very 
different from a question of counting matrices from $\cM_n\(\Z; H\)$ with an irreducible characteristic 
polynomial, which can be treated via explicit versions of the Hilbert Irreducibility Theorem, 
see,~for example,~\cite{B-SGa} and references therein. 

\section{Families of matrices with explicit characteristic  polynomials}
\label{sec: Set F}
Let $\cB\subseteq [0,H]^{(2k+1)\times (2k+1)}$ be the set of lower Hessenberg matrices given by the matrices
$$
B=(b_{i,j})_{i\in [-k,k], \, j\in [-k,k]}
$$
(where it is convenient to run both indices from $-k$ to $k$) such that 
\begin{itemize}
    \item $
    b_{i,i+1}=
    \left\{
\begin{array}{cl}
   1  & \text{for } i\in [-k,0];\\
   H  & \text{for } i\in [1,k-1];
\end{array}
    \right.
    $ 
\item $b_{i,j}\in \{0,1,\dots, H-1\}$ for $i\in [1,k]$ and $j\in [-k,-1]$;  
\item all other terms are zero.
\end{itemize}

That is, $B$ has
\begin{itemize}
\item zeros on the main diagonal;
\item $k+1$ values of $1$ and then $k-1$ values of $H$ 
on the diagonal above the main diagonal;
\item an arbitrary $k\times k$ matrix in left bottom corner, with entries from $\{0,1,\dots, H-1\}$.
\end{itemize}


Define the following set $\cF\subseteq \Z[X]$ of polynomials of degree $n = 2k+1$: 
$$
f = X^{2k+1}-f_{2k-2}X^{2k-2} - f_{2k-3}X^{2k-3}-\dots - f_1X-f_0
$$
with coefficients satisfying the following conditions on their size and 
divisibility: 
\begin{itemize}
    \item $f_{2k}=f_{2k-1}=0$,
    \item $f_{2k-j}\in\{0,1,\dots, H^{j-1}-1 \}$ for $j\in [2,k+1]$,
    \item $f_{k-j}/H^{j-1}\in\{0,1,\dots,H^{k-j+1}-1\}$ for $j\in [2,k]$.
\end{itemize}

We have, see~\cite[Theorem~3.1]{ALPS}, the following one-to-one correspondence. 
Let $\charpoly(M)$ denote the characteristic polynomial of  matrix $M \in  \Z^{n\times n}$. Then we have
the following connection between $\cB$ and $\cF$.

\begin{lemma}\label{lem: B<->F}
The map $\charpoly: \Z^{n \times n}\to  \Z[X]$
acts as a bijection from $\cB$ to $\cF$.
\end{lemma}

\section{Irreducible polynomials over finite fields}

Let $\pi_q(n)$ denote the number of monic irreducible polynomials of degree $n$ 
over the finite field $\F_q$ of $q$ elements.  It is well know that  $\pi_q(n)$ 
is of order $q^n/n$. We however require some explicit bounds, which are conveniently 
presented by Pollack~\cite[Lemma~4]{Pol}.

\begin{lemma}\label{lem:PNT-Fq}
 For each prime power $q$ and each integer $n\ge 1$, we have
$$
\frac{q^n}{n} -  2\frac{q^{n/2}}{n}  \le \pi_q(n)\le \frac{q^n}{n}.
$$
\end{lemma}

For completeness we presented in Lemma~\ref{lem:PNT-Fq} both-sided inequalities,
however for us only the lower bound is important. 

We also need the following special case of a much more general  result of Pollack~\cite[Theorem~1]{Pol}.

\begin{lemma}\label{lem:Pollack}
Let $\pi_q^*(n)$ be the number of irreducible polynomials of degree $n$ 
over  $\F_q$ of the form
$$
X^n + a_{n-3} X^{n-3} + \ldots + a_1X + a_0 \in \F_q[X].
$$
 Then, for each prime power $q$ and each integer $n\ge 1$, we have
$$
\left|\pi_q^*(n) -  q^{-2} \pi_q(n)\right|\le q^{n - \fl{n/2}/2} +  q^{n -1 - \fl{n/3}} .
$$
\end{lemma}

Taking $q = 2$ in Lemma~\ref{lem:Pollack} we derive.

\begin{cor}\label{cor:Irred F2}
For any integer $n\ge 4$ we have 
$$
\pi_2^*(n)  \ge \frac{2^n}{16n}. 
$$
\end{cor}

\begin{proof} From Lemmas~\ref{lem:PNT-Fq} and~\ref{lem:Pollack} we derive 
$$
\pi_2^*(n)  \ge \frac{2^n}{4n} -   \frac{2^{n/2+1}}{4 n}  
- 2^{n - \fl{n/2}/2} - 2^{n -1 - \fl{n/3}}
$$
and for $n \ge 30$ the result is immediate from the bound.  

For $4 \le n \le 29$ we compute $\pi_2^*(n)$ directly and verify the 
desired bound.
\end{proof}

\section{Proof of Theorem~\ref{thm:InH-asym}}

Let  
\begin{equation}
\label{eq:p<H}
p < H
\end{equation}
be a prime to be chosen later and let $\cJ_p\subseteq \F_p[X]$ the set of irreducible polynomial of degree $n = 2k+1$ with zero coefficients of the terms $X^{2k}$ and $X^{2k-1}$. 
Note that  since $n \ge 11$ we have $k \ge 5$.

By Lemmas~\ref{lem:PNT-Fq} and~\ref{lem:Pollack}, we have
\begin{equation}
\label{eq:Card Jp}
\#\cJ_p = \pi_p^*(n)= \frac{p^{2k-1}}{2k+1} + O\left(p^{3k/2+1}\right) 
\end{equation}
with some absolute implied constant. 

We lift all $g\in \cJ_p$ into $\cF$. For every $j\in [2,k+1]$, the number of choices for 
$$
f_{2k-j}\equiv g_{2k-j} \mod p \mand f_{2k-j}\in\{0,1,\dots, H^{j-1}-1\} 
$$
is at least  
\begin{equation}
\label{eq:Coeff f_j}
\fl{H^{j-1}/p}\geq  H^{j-1}/p -1.
\end{equation}
Since $p\nmid H$,  for every $j\in [2,k]$, the number of choices for 
$$
f_{k-j}\equiv g_{k-j} \mod p   \mand \frac{f_{k-j}}{H^{j-1}}\in\{0,1,\dots, H^{k-j+1}-1\} 
$$
is at least 
\begin{equation}
\label{eq:Coeff f_2k-j}
\ccr{\fl{H^{k-j+1}/p}\geq H^{k-j+1}/p -1.}
\end{equation}
Thus for a fixed $g\in\cJ_p$, the total number of the above polynomials 
$f \in \Z[X]$ to which it can be lifted is at least
\begin{align*}
\prod_{j=2}^{k+1}& \frac{H^{j-1}-\ccr{p}}{p} \cdot  \prod_{j=2}^{k}\frac{H^{k-j+1}-\ccr{p}}{p} \\
&= \frac{1}{p^{2k-1}}H^{\frac{k(k+1)}{2}+ \frac{k(k-1)}{2}} \prod_{j=2}^{k+1}
\left(1-\frac{\ccr{p}}{H^{j-1}}\right)  \cdot  \prod_{j=2}^{k}\left(1-\frac{\ccr{p}}{H^{k-j+1}}\right) \\
& =\frac{H^{k^2}}{p^{2k-1}}  \(1+O\(kp/H\)\)  
\end{align*}  
\ccr{(we emphasise that  the implied constant is absolute here). }
Thus, by Lemma~\ref{lem: B<->F} and the equation~\eqref{eq:Card Jp}, we have 
\begin{align*}
 \# \cI_n^+(H)  &\ge  \#\left\{\charpoly(M):~M\in [0,H]^{n\times n} \right\}\\
&  \geq \#\cJ_p \cdot \frac{H^{k^2}}{p^{2k-1}}  \(1+O\(k\frac{p}{H}\)\) \\
& \geq \(\frac{p^{2k-1}}{2k+1} + O\(p^{3k/2+1}\) \)  \cdot \frac{H^{k^2}}{p^{2k-1}}  \(1+O\(k\frac{p}{H}\)\) \\
&  \geq 
  \frac{H^{k^2}}{2k+1} +O\(pH^{k^2-1}+\frac{H^{k^2}}{p^{k/2-2}} +
  k \frac{H^{k^2-1}}{p^{k/2-3}} \).
\end{align*}

We now choose  $p$ as the smallest prime with $p > 2 H^{2/(k-2)}$
and also  observe that   by the prime number theorem we have $p = O\(H^{2/(k-4)}\)$.
\ccr{Adjusting the absolute constant in the desired bound, we can   assume that $H$ is large enough 
so that such smallest prime $p$ satisfies~\eqref{eq:p<H} as well. }
Thus, this choice  makes the first \ccr{error} term in the above inequality dominate
the other two \ccr{error} terms. After simple calculations, we obtain the result.

\section{Proof of Theorem~\ref{thm:InH-expl}}
We now apply the previous argument with $p=2$, but we additionally note that since $H$ is odd then 
instead \ccr{of~\eqref{eq:Coeff f_j} and~\eqref{eq:Coeff f_2k-j} we obtain 
$$
\fl{H^{j-1}/2} = (H^{j-1}-1)/2, \qquad j\in [2,k+1],
 $$ 
 and
 $$
  \fl{H^{k-j+1}/2} = (H^{k-j+1}-1)/2, \qquad j\in [2,k],
$$
respectively.}
We also use Corollary~\ref{cor:Irred F2} instead of~\eqref{eq:Card Jp}.
Thus, we obtain 
\begin{align*}
 \# \cI_n^+(H)  &\ge  \#\left\{\charpoly(M):~M\in [0,H]^{n\times n}, \ \ccr{\charpoly(M) \ \text{irred.}} \right\}\\
&  \geq \frac{2^{2k+1}}{\ccr{16}(2k+1)}  \cdot
\prod_{j=2}^{k+1} \frac{H^{j-1}-1}{2} \cdot  \prod_{j=2}^{k}\frac{H^{k-j+1}-1}{2} \\
&  \ccg{=} \frac{1}{\ccr{4}(2k+1)}  \cdot
\prod_{j=2}^{k+1} \(H^{j-1}-1\) \cdot  \prod_{j=2}^{k} \(H^{k-j+1}-1\) \\
& = \frac{1}{\ccr{4}(2k+1)}  \cdot  \(H^{k}-1\) 
\prod_{j=1}^{k-1} \(H^{j}-1\) ^2 \\
&  \geq \frac{1}{\ccr{4}(2k+1)}  \cdot  \(H-1\)^k
\prod_{j=1}^{k-1} \(H -1\) ^{2j}  , 
\end{align*}
and the result follows.

\section{Comments}  
We remark that when $H$ is exponentially larger than $n$, 
more precisely, if 
\begin{equation}
\label{eq:Small n}
n \le (1-\varepsilon) \frac{\log H}{\log 2}
\end{equation}
for a fixed $\varepsilon > 0$, then 
a result of Pham and  Xu~\cite[Theorem~1.2] {PhXu} shows that in fact  
almost polynomials in the set $\cF$ in Section~\ref{sec: Set F} are irreducible. Thus
$$
 \# \cI_n^+(H) \ge  (1+o(1)) H^{(n-1)^2/4} ,
 $$
 provided as $H \to \infty$ under the condition~\eqref{eq:Small n}.

\section*{Acknowledgements}  

The authors would like to thank Lior Bary-Soroker  for  very useful discussions and for 
informing them about the work of Pham and  Xu~\cite{PhXu}. 

During the preparation of this  work L.~M. was partially supported by NRDI (National Research Development and Innovation Office, Hungary) grant 
FK142960, and by
the János Bolyai Research Scholarship of the Hungarian Academy of
Sciences; and
I.~S. by the Australian Research Council Grants  DP230100530 and DP230100534.


\begin{thebibliography}{99}


\bibitem{ALPS}
A. Abrams, Z. Landau, J. Pommersheim and N. Srivastava, 
`On eigenvalue gaps of integer matrices',  
{\it Math.  Comp.\/}, {\bf 94} (2025), 853--862.


\bibitem{B-SGa} L. Bary-Soroker and D. Garzoni,  
`Hilbert's Irreducibility Theorem via random walks', 
{\it Int. Math. Res. Not.\/}, {\bf 2023} (2023),  12512--12537. 



\bibitem{HOS} P. Habegger, A. Ostafe and I. E. Shparlinski,  `Integer matrices with a given characteristic polynomial and multiplicative dependence of matrices',
{\it Preprint\/}, 2022, available  from \url{https://arxiv.org/abs/2203.03880}

\bibitem{PhXu} H.  T.  Pham and M.  W. Xu, `Irreducibility of random polynomials of bounded degree', 
{\it Discrete Analysis\/}, 2021:7, 1--16, available  from \url{www.discreteanalysisjournal.com}.


\bibitem{Pol} P.  Pollack, 
`Irreducible polynomials with several prescribed coefficients'
{\it Finite Fields and Appl.\/}, {\bf  22} (2013), 70--78.


\end{thebibliography}
\end{document}